\documentclass{amsart}

\usepackage{amssymb,amsfonts, amsmath, amsthm}
\usepackage{xypic}
\usepackage[active]{srcltx}

\numberwithin{equation}{section}

\theoremstyle{plain}

\newtheorem{theorem}{Theorem}[section]
\newtheorem{proposition}[theorem]{Proposition}
\newtheorem{lemma}[theorem]{Lemma}
\newtheorem{corollary}[theorem]{Corollary}

\theoremstyle{definition}

\newtheorem{defn}[theorem]{Definition}
\newtheorem{example}[theorem]{Example}
\newtheorem{remark}[theorem]{Remark}

\newcommand{\comment}[1]{}

\def\Hom{\textup{Hom}^+}
\def\Ind{\textup{Ind}}
\def\Aut{\textup{Aut}^+}
\def\cS{{\mathcal{S}}}
\newcommand{\rt}{\rtimes}
\newcommand{\R}{{\mathbb R}}
\newcommand{\Z}{{\mathbb Z}}

\newcommand{\N}{{\mathbb N}}

\newcommand{\supp}{{\operatorname {supp}\,}}

\newcommand{\Rn}{(\R_{>0})^n}
\newcommand{\AutRn}{\Aut(\R^n)}
\newcommand{\id}{\textup{id}}

\begin{document}

\title[]
{{Positive representations of finite groups in Riesz spaces}}
\author{Marcel de Jeu}
\email{mdejeu@math.leidenuniv.nl}
\address{Mathematical Institute,
Leiden University,
P.O. Box 9512,
2300 RA Leiden,
The Netherlands}
\author{Marten Wortel}
\email{marten.wortel@gmail.com}
\address{Mathematical Institute,
Leiden University,
P.O. Box 9512,
2300 RA Leiden,
The Netherlands}

\begin{abstract}
In this paper, which is part of a study of positive representations of locally compact groups in Banach lattices, we initiate the theory of positive representations of finite groups in Riesz spaces. If such a representation has only the zero subspace and possibly the space itself as invariant principal bands, then the space is Archimedean and finite dimensional. Various notions of irreducibility of a positive representation are introduced and, for a finite group acting positively in a space with sufficiently many projections, these are shown to be equal. We describe the finite dimensional positive Archimedean representations of a finite group and establish that, up to order equivalence, these are order direct sums, with unique multiplicities, of the order indecomposable positive representations naturally associated with transitive $G$-spaces. Character theory is shown to break down for positive representations. Induction and systems of imprimitivity are introduced in an ordered context, where the multiplicity formulation of Frobenius reciprocity turns out not to hold.

\end{abstract}

\maketitle

\section{Introduction and overview}

The theory of unitary group representations is well developed. Apart from its intrinsic appeal, it has been stimulated in its early days by the wish, originating from quantum theory, to understand the natural representations of symmetry groups of physical systems in $L^2$-spaces. Such symmetry groups do not only yield natural unitary representations, but they have natural representations in other Banach spaces as well. For example, the orthogonal group in three dimensions does not only act on the $L^2$-functions on the sphere or on three dimensional space. It also has a natural isometric action on $L^p$-functions, for all $p$, and this action is strongly continuous for finite $p$. Moreover, this action is obviously positive. Thus, for finite $p$, these $L^p$-spaces, which are Banach lattices, afford a strongly continuous isometric positive representation of the orthogonal group. It is rather easy to find more examples of positive representations: whenever a group acts on a point set, then, more often than not, there is a natural positive action on various naturally associated Banach lattices of functions.

However, in spite of the plenitude of examples of strongly continuous (isometric) positive representations of groups in Banach lattices, the theory of such representations cannot compare to its unitary counterpart. Very little seems to be known. Is there, for example, a decomposition theory into indecomposables for such representations, as a positive counterpart to that for unitary representations described in, e.g., \cite{dixmier}?  When asking such an---admittedly ambitious---question, it is important to keep in mind that the unitary theory works particularly well in separable Hilbert spaces, i.e., for representations which can all be realized in just one space: $\ell^2$. Since there is a great diversity of Banach lattices, it is not clear at the time of writing whether one can expect a general answer for all these lattices with a degree of sophistication and uniformity comparable to that for the unitary representations in this single Hilbert space $\ell^2$. It may be more feasible, at least for the moment, to aim at a better understanding of positive representations on specific classes of Banach lattices. For example, the results in \cite{dejeurozendaal} show that, in the context of a Polish group acting on a Polish space with an invariant measure, it is indeed possible to decompose---in terms of Banach bundles rather than in terms of direct integrals as in the unitary case---the corresponding isometric positive representation in $L^p$-spaces ($1\leq p<\infty)$ into indecomposable isometric positive representations. To our knowledge, this is the only available decomposition result at this time. Since this result covers only representations originating from an action on the underlying point set, we still cannot decide whether a general (isometric) positive representation of a (Polish) group in such Banach lattices can be decomposed into indecomposable positive representations, and more research is necessary to decide this. This, however, is already a relatively advanced issue: as will become clear below, it is easy to ask very natural basic questions about positive representations of locally compact groups in Banach lattices which need answering. This paper, then, is a contribution to the theory of such representations, with a hoped-for decomposition theory into indecomposable positive representations in mind as a leading and focusing theme, and starting with the obviously easiest case: the finite groups. We will now globally discuss it contents.

If a group $G$ acts as positive operators on a Banach lattice $E$, then the natural question is to ask whether it is possible to decompose $E=L\oplus M$ as a $G$-invariant order direct sum. In that case, $L$ and $M$ are automatically projection bands and each other's disjoint complement. Since bands in a Banach lattice are closed, the decomposition is then automatically also topological, and the original representations splits as an order direct sum of the positive subrepresentations on the Banach lattices $L$ and $M$. If such a decomposition is only trivially possible, then we will call the representation (order) indecomposable, a terminology already tacitly used above. Is it then true that every indecomposable positive representation of a finite group $G$ in a Banach lattice is finite dimensional, as it is for unitary representations? This is not the case: the trivial group acting on $C([0,1])$, which has only trivial projection bands, provides a counterexample. Is it then perhaps true when we narrow down the class of spaces to better behaved ones, and ask the same question for Banach lattices with the projection property? After all, since bands are now complemented by their disjoint complement, they seem close to Hilbert spaces where the proof for the corresponding statement in the unitary case is a triviality, and based on this complementation property. Indeed, if $x\neq 0$ is an element of the Hilbert space under consideration where the finite group $G$ acts unitarily and irreducibly, then the orbit $G\cdot x$ spans a finite dimensional, hence closed, nonzero subspace which is clearly invariant and invariantly complemented. Hence the orbit spans the space, which must be finite dimensional. In an ordered context this proof breaks down. Surely, one can consider the band generated by the orbit of a nonzero element, which is invariantly complemented in an order direct sum if the space is assumed to have the projection property. Hence this band is equal to the space, but since there is no guarantee that it is finite dimensional, once cannot reach the desired conclusion along these lines. We have not been able to find an answer to this finite dimensionality question for Banach lattices with the projection property in the literature, nor could a number of experts in positivity we consulted provide an answer. The best available result in this vein seems to be \cite[Theorem~III.10.4]{schaefer}, which implies as a special case that a positive representation of a finite group in a Banach lattice with only trivial invariant closed ideals is finite dimensional. Still, this does not answer our question concerning the finite dimensionality of indecomposable positive representations of a finite group in a Banach lattice with the projection property. The reason is simple: unless one assumes that the lattice has order continuous norm, one cannot conclude that there are only trivial invariant closed ideals from the fact that there are only trivial invariant bands. On the other hand: there are no obvious infinite dimensional counterexamples in sight, and one might start to suspect that there are none. This is in fact the case, and even more holds true: a positive representation of a finite group in a Riesz space, with the property that the only invariant principal bands are $\{0\}$ and possibly the space itself, is in a finite dimensional Archimedean space, cf.\ Theorem~\ref{t:principal_band_irreducible_finite_dim} below. As will have become obvious from the previous discussion, such a result is no longer a triviality in an ordered context. It provides an affirmative answer to our original question because, for Banach lattices with the projection property, an invariant principal band is an invariant projection band. It also implies the aforementioned result that a positive representation of a finite group in a Banach lattice with only trivial invariant closed ideals is finite dimensional. Indeed, an invariant principal band is then an invariant closed ideal.

Thus, even though our original question was in terms of Banach lattices, and motivated by analytical unitary analogies, an answer can be provided in a more general, topology free context. For finite groups, this is---after the fact---perhaps not too big a surprise. Furthermore, we note that the hypothesis in this finite dimensionality theorem is not the triviality of all invariant order decompositions, but rather the absence of a nontrivial $G$-invariant object, without any reference to this being invariantly complemented in an order direct sum or not. It thus becomes clear that it is worthwhile to not only consider the existence of nontrivial invariant projection bands (which is the same as the representation being (order) decomposable), but to also consider the existence of nontrivial invariant ideals, nontrivial invariant bands, etc., for positive representations in arbitrary Riesz spaces, and investigate the interrelations between the corresponding notions of irreducibility. In the unitary case, indecomposability and irreducibility (for which there is only one reasonable notion) coincide, but in the present ordered context this need not be so. Nevertheless, for finite groups acting positively in spaces with sufficiently many projections, the most natural of these notions of irreducibility are all identical and coincide with (order) indecomposability, cf.\ Theorem~\ref{t:equivalencies}. Again, whereas the corresponding proof for the unitary case is a triviality, this is not quite so obvious in an ordered context.

As will become apparent in this paper, it is possible to describe all finite dimensional positive Archimedean representations of a finite group, indecomposable or not. Once this is done, it is not too difficult to show that such finite dimensional positive representations can be decomposed uniquely into indecomposable positive representations, and that, up to order equivalence, all such indecomposable positive representations arise from actions of the group on transitive $G$-spaces, cf.\ Corollary~\ref{c:decomposition_with_multiplicities} and Theorem~\ref{t:order_dual}. Since this decomposition into irreducible positive representations with multiplicities is so reminiscent of classical linear representation theory theory for finite groups, and to Peter-Weyl theory for compact groups, one might wonder whether parts of character theory also survive. This is hardly the case. For finite groups with only normal subgroups, such as finite abelian groups, there is still a bijection between characters and order equivalence classes of finite dimensional indecomposable positive Archimedean representations, cf.\ Corollary~\ref{c:dedekindgroups}, but we provide counterexamples to a number of other results as they would be natural to conjecture.

Finally, we consider induction and systems of imprimitivity in an ordered context. As long as topology is not an issue, this can be done from a categorical point of view for arbitrary groups and arbitrary subgroups. Even though the constructions are fairly routine, we have included the material, not only as a preparation for future more analytical considerations, but also because there are still some differences with the linear theory. For example, Frobenius reciprocity no longer holds in its multiplicity formulation.

After this global overview we emphasize that, even though this paper contains a basic finite dimensionality result and provides reasonably complete results for finite dimensional positive Archimedean representations of finite groups (in analytical terms: for positive representations of finite groups in $C(K)$ for $K$ finite), the basic decomposition issue for positive representations of a finite group in infinite dimensional Banach lattices is still open. At the time of writing, the only results in this direction seem to be the specialization to finite groups of the results in \cite{dejeurozendaal} for $L^p$-spaces, and of those in \cite{dejeumesserschmidt}, which is concerned with Jordan-H\"older theory for finite chains of various invariant order structures in Riesz spaces. As long as the group is only finite, a more comprehensive answer seems desirable.

\medskip

The structure of this paper is as follows.

In Section~\ref{s:Preliminaries} we introduce the necessary notation and definitions, and we recall a folklore result on transitive $G$-spaces. Then, in Section~\ref{s:Irreducible and indecomposable representations}, we investigate the relations between various notions of irreducibility as already mentioned above. We then establish one of the main results of this paper, Theorem~\ref{t:principal_band_irreducible_finite_dim}, stating, amongst others, that a principal band irreducible positive representation of a finite group is always finite dimensional. The proof is by induction on the order of the group, and uses a reasonable amount of general basic theory of Riesz spaces.
We then continue by examining the structure of finite dimensional positive Archimedean representations of a finite group in Section~\ref{s:Structure of finite dimensional representations}. Any such space is lattice isomorphic to $\R^n$, for some $n$, and its group $\Aut(\R^n)$ of lattice automorphisms is a semidirect product of $S_n$ and the group of multiplication (diagonal) operators, a result which also follows from \cite[Theorem~3.2.10]{meyernieberg}, but which we prefer to derive by elementary means in our context. Armed with this information we can completely determine the structure of a positive representations of a finite group in $\R^n$ in Theorem~\ref{t:characterization_rep_into_Rn}: such a positive representation is given by a representation into $S_n \subset \Aut(\R^n)$, called a permutation representation, and a single multiplication operator. We then determine when two positive representations in $\R^n$ are order equivalent, which turns out to be the case precisely when their permutation parts are conjugate. Consequently, in the end, the finite dimensional positive Archimedean representations of a finite group can be described in terms of actions of the group on finite sets. The decomposition result and the description of indecomposable positive representations already mentioned above then follow easily.  The rest of Section~\ref{s:Structure of finite dimensional representations} is concerned with showing that character theory does not survive in an ordered context. Finally, in Section~\ref{s:Induction}, we develop the theory of induction and systems of imprimitivity in the ordered setting, and show that Frobenius reciprocity does not hold in its multiplicity formulation.

\section{Preliminaries}\label{s:Preliminaries}

In this section we will discuss some preliminaries about automorphisms of Riesz spaces, representations, order direct sums of representations and $G$-spaces. All Riesz spaces in this paper are real. For positive representations in spaces admitting a complexification it is easy, and left to the reader, to formulate the corresponding complex result and derive it from the real case.

Let $E$ be a not necessarily Archimedean Riesz space. If $D \subset E$ is any subset, then the band generated by $D$ is denoted by $\{D\}$, and the disjoint complement of $D$ is denoted by $D^d$. If $T$ is a lattice automorphism of $E$, then $\{TD\} = T\{D\}$ and $T(D^d) = (TD)^d$. The group of lattice automorphisms of $E$ is denoted by $\Aut(E)$.

In this paper $\R^n$ is always equipped with the coordinatewise ordering.

\begin{defn}
 Let $G$ be a group and $E$ a Riesz space. A \emph{positive representation} of $G$ in $E$ is a group homomorphism $\rho \colon G \to \Aut(E)$.
\end{defn}
For typographical reasons, we will write $\rho_s$ instead of $\rho(s)$, for $s \in G$.

If $(E_i)_{i \in I}$ is a collection of Riesz spaces, then the \emph{order direct sum} of this collection, denoted $\oplus_{i \in I} E_i$, is the Riesz space with elements $(x_i)_{i \in I}$, where $x_i \in E_i$ for all $i \in I$, at most finitely many $x_i$ are nonzero, and where $(x_i)_{i \in I}$ is positive if and only if $x_i$ is positive for all $i \in I$. If additionally $\rho^i \colon G \to \Aut(E_i)$ is a positive representation for all $i \in I$, then the positive representation
\[ \bigoplus_{i \in I} \rho^i \colon \to \Aut\left(\bigoplus_{i \in I} E_i \right), \]
the \emph{order direct sum} of the $\rho^i$, is defined by $(\oplus_{i \in I} \rho^i)_s := \oplus_{i \in I} \rho^i_s$, for $s \in G$.

Let $\rho \colon G \to \Aut(E)$ be a positive representation, and suppose that $E = L \oplus M$ as an order direct sum. Then $L$ and $M$ are automatically projection bands with $L^d = M$ by \cite[Theorem~24.3]{riesz1}. If both $L$ and $M$ are $\rho$-invariant, then $\rho$ can be viewed as the order direct sum of $\rho$ acting positively on $L$ and $M$.

If $\rho \colon G \to \Aut(E)$ and $\theta \colon G \to \Aut(F)$ are positive representations on Riesz spaces $E$ and $F$, respectively, then a positive map $T \colon E \to F$ is called a \emph{positive intertwiner} between $\rho$ and $\theta$ if $T \rho_s = \theta_s T$ for all $s \in G$, and $\rho$ and $\theta$ are called \emph{order equivalent} if there exists a positive intertwiner between $\rho$ and $\theta$ which is a lattice isomorphism.

\medskip

Turning to the terminology for $G$-spaces, we let $G$ be a not necessarily finite group. A $G$-space $X$ is a nonempty set $X$ equipped with an action of $G$; it is called transitive if there is only one orbit. For $x \in X$, let $G_x$ denote the subgroup $\{s \in G: sx = x\}$, the stabilizer of $x$. If $X$ and $Y$ are $G$-spaces, then $X$ and $Y$ are called isomorphic $G$-spaces if there is a bijection $\phi \colon X \to Y$, such that $s \phi(x) = \phi(sx)$ for all $s \in G$ and $x \in X$. We let $[X]$ denote the class of all $G$-spaces isomorphic to $X$.

If $X$ is a transitive $G$-space and $x \in X$, then $sx \mapsto sG_x$ is a $G$-space isomorphism between $X$ and $G / G_x$ with its natural $G$-action. The next folklore lemma elaborates on this correspondence.

\begin{lemma}\label{t:characterization_transitive_actions}
Let $G$ be a not necessarily finite group. For each isomorphism class $[X]$ of transitive $G$-spaces, choose a representative $X$ and $x \in X$. Then the conjugacy class $[G_x]$ of $G_x$ does not depend on the choices made, and the map $[X] \mapsto [G_x]$ is a bijection between the isomorphism classes of transitive $G$-spaces and the conjugacy classes of subgroups of $G$.
\end{lemma}

\begin{proof}
It is easy to see that the map is well-defined and surjective. For injectivity, let $X$ and $Y$ be transitive $G$-spaces, such that $[G_x] = [G_y]$ for some $x \in X$ and $y \in Y$. We have to show that $[X] = [Y]$, or equivalently, $G / G_x \cong G / G_y$. By assumption $G_x = r G_y r^{-1}$ for some $r \in G$, and the map $s G_x \mapsto s G_x r = sr G_y$ is then an isomorphism of $G$-spaces between $G / G_x$ and $G / G_y$.
\end{proof}

\section{Irreducible and indecomposable representations}\label{s:Irreducible and indecomposable representations}

In the theory of unitary representations of groups, the nonexistence of nontrivial closed invariant subspaces is the only reasonable notion of irreducibility of a representation, and it coincides with the natural notion of indecomposability of a representation. In a purely linear context, irreducibility and indecomposability of group representations need not coincide, however, and the same is true in an ordered context where, in addition, there are several natural notions of irreducibility. In this section, we are concerned with the relations between the various notions and we establish a basic finite dimensionality result, Theorem~\ref{t:principal_band_irreducible_finite_dim}. This is then used to show that, after the fact, the various notions of irreducibility are equivalent for finite groups if the space has sufficiently many projections, cf.\ Theorem~\ref{t:equivalencies}. We let $G$ be a group, to begin with not necessarily finite.

\begin{defn}
 A positive representation $\rho \colon G \to \Aut(E)$ is called \emph{band irreducible} if a $\rho$-invariant band equals $\{0\}$ or $E$. Projection band irreducibility, ideal irreducibility, and principal band irreducibility are defined similarly, as are closed ideal irreducibility, etc., in the case of normed Riesz spaces.
\end{defn}

Starting our discussion of the implications between the various notions of irreducibility, we note that, obviously, band irreducibility implies projection band irreducibility. If $E$ has the projection property, then the converse holds trivially as well, since all bands are projection bands by definition, but the next example shows that this converse fails in general.

\begin{example}\label{e:ex1}
Consider the representation of the trivial group on $C[0,1]$. Now every band is invariant, so this representation is not band irreducible, but $C[0,1]$ does not have any nontrivial projection bands, and therefore it is projection band irreducible.
\end{example}

For positive representations on Banach lattices, closed ideal irreducibility implies band irreducibility. If the Banach lattice  has order-continuous norm, then the converse holds as well, since then all closed ideals are bands (\cite[Corollary~2.4.4]{meyernieberg}), but once again this converse fails in general, as the next example shows.

\begin{example}
Consider $\ell^\infty(\Z)$, the space of doubly infinite bounded sequences, and define $\rho \colon \Z \to \Aut(\ell^\infty(\Z))$ by $\rho_k (x_n) := (x_{n-k})$, the left regular representation. This representation is not closed ideal irreducible, since the space $c_0(\Z)$ of sequences tending to zero is an invariant closed ideal. On the other hand, it is easy to see that any nonzero invariant ideal must contain the order dense subspace of finitely supported sequences, therefore $\rho$ is band irreducible.
\end{example}

Finally, for positive representations on Banach lattices, ideal irreducibility obviously implies closed ideal irreducibility, but again there is an example showing that the converse fails in general.

\begin{example}
 Consider the left regular representation of $\Z$ on $c_0(\Z)$, as in the above example. Then $\ell^1(\Z)$ is an invariant ideal, so this positive representation is not ideal irreducible, but it is closed ideal irreducible since every nonzero invariant ideal must contain the norm dense subspace of finitely supported sequences.
\end{example}

We continue by defining the natural notion of indecomposability for positive representations, which is order indecomposability. As usual, the order direct sum $E = L \oplus M$ is called nontrivial if $L \not= 0$ and $L \not= E$.

\begin{defn}
 A positive representation $\rho \colon G \to \Aut(E)$ is called \emph{order indecomposable} if there are no nontrivial $\rho$-invariant order direct sums $E = L \oplus M$.
\end{defn}

We will now investigate the conditional equivalence between order indecomposability and the various notions of irreducibility.

\begin{lemma}\label{l:projband_equiv_order_irr}
 A positive representation $\rho \colon G \to \Aut(E)$ is order indecomposable if and only if it is projection band irreducible.
\end{lemma}

\begin{proof}
 Suppose $\rho$ is order indecomposable, and let $B$ be a $\rho$-invariant projection band. We claim that $B^d$ is $\rho$-invariant. For this, let $x \in (B^d)^+$ and $s \in G$. Then $(\rho_s x) \wedge y = \rho_s (x \wedge \rho_s^{-1} y) = \rho_s 0 = 0$ for all $y \in B^+$, so $\rho_s x \bot B$, i.e., $\rho_s x \in B^d$. Hence $E = B \oplus B^d$ is a $\rho$-invariant order direct sum, so either $B = 0$ or $B = E$.

Conversely, suppose that $\rho$ is projection band irreducible. Let $E = L \oplus M$ be a $\rho$-invariant order direct sum. Then, as mentioned in the preliminaries, $L$ and $M$ are projection bands, and therefore $L = 0$ or $M = 0$.
\end{proof}

Thus order indecomposability is equivalent with projection band irreducibility. We have already seen that the latter property is, in general, not equivalent with band irreducibility, but that equivalence between these two does hold (trivially) if the Riesz space has the projection property. However, if the group is finite, we will see in Lemma~\ref{l:projband_equiv_band_irr} and Theorem~\ref{t:equivalencies} below that these three notion are equivalent under a much milder assumption on the space, as in the following definition.

\begin{defn}
 A Riesz space $E$ is said to have \emph{sufficiently many projections} if every nonzero band contains a nonzero projection band.
\end{defn}

This notion is intermediate between the principal projection property and the Archimedean property, cf. \cite[Theorem~30.4]{riesz1}. In order to show that it is (in particular) actually weaker than the projection property, which is the relevant feature for our discussion, we present an example of a Banach lattice which has sufficiently many projections, but not the projection property.

\begin{example}
Let $\Delta \subset [0,1]$ be the Cantor set, and let $E = C(\Delta)$. Then \cite[Corollary~2.1.10]{meyernieberg} shows that bands correspond to (all functions vanishing on the complement of) regularly open sets, i.e., to open sets which equal the interior of their closure, and projection bands correspond to clopen sets. The Cantor set has a basis of clopen sets, so that, in particular, every nonempty regularly open set contains a nonempty clopen set. Therefore $C(\Delta)$ has sufficiently many projections. It does not have the projection property, since $[0,1/4) \cap \Delta \subset \Delta$ is regularly open but not closed (\cite[29.7]{counterexamples}).
\end{example}

\begin{lemma}\label{l:projband_equiv_band_irr}
 Let $G$ be a finite group, $E$ a Riesz space with sufficiently many projections, and $\rho \colon G \to \Aut(E)$ a positive representation. Then the following are equivalent:
\begin{enumerate}
 \item $\rho$ is order indecomposable;
 \item $\rho$ is projection band irreducible;
 \item $\rho$ is band irreducible.
\end{enumerate}
\end{lemma}

\begin{proof}
$(iii) \Rightarrow (ii)$ is immediate. For $(ii) \Rightarrow (iii)$, suppose $\rho$ is projection band irreducible. Let $B_0$ be a nonzero $\rho$-invariant band. Let $0 \not= B \subset B_0$ be a projection band. Then $\sum_{s \in G} \rho_s B$ is a projection band by \cite[Theorem~30.1(ii)]{riesz1}, and clearly it is $\rho$-invariant, nonzero, and contained in $B_0$. Therefore it must equal $E$, and so $B_0 = E$.

$(i) \Leftrightarrow (ii)$ follows from Lemma~\ref{l:projband_equiv_order_irr}.
\end{proof}

This lemma will be improved significantly later on, see Theorem~\ref{t:equivalencies}.

We will now investigate the question whether a positive representation of a finite group, satisfying a suitable notion of irreducibility, is necessarily finite dimensional. As explained in the introduction, this is not quite as obvious as it is for Banach space representations. It follows as a rather special case from \cite[Theorem~III.10.4]{schaefer} that a closed ideal irreducible positive representation of a finite group in a Banach lattice is finite dimensional, but this seems to be the only known available result in this vein. We will show, see Theorem~\ref{t:principal_band_irreducible_finite_dim}, that a positive principal band irreducible representation of a finite group in a Riesz space is finite dimensional. This implies the aforementioned finite dimensionality result for Banach lattices. As a preparation, we need four lemmas.

\begin{lemma}\label{l:irrarch}
Let $G$ be a finite group, $E$ a Riesz space and $\rho \colon G \to \Aut(E)$ a positive principal band irreducible representation. Then $E$ is Archimedean.
\end{lemma}

\begin{proof}
 Suppose $E$ is not Archimedean. Then $E\neq 0$ and there exist $x,y \in E$ such that $0 < \lambda x \leq y$ for all $\lambda > 0$. The band $B$ generated by $\sum_{s \in G} \rho_s x$ is principal, $\rho$-invariant and nonzero, and therefore equals $E$. Let $u \geq 0$ be an element of the ideal $I$ generated by $\sum_{s \in G} \rho_s x$. Then for some $\lambda \geq 0$,
\[ u \leq \lambda \sum_{s \in G} \rho_s x = \sum_{s \in G} \rho_s (\lambda x) \leq \sum_{s \in G} \rho_s y, \]
and so $\sum_{s \in G} \rho_s y$ is an upper bound for $I^+$, and hence for $B^+ = E^+$, which is absurd since $E\neq 0$. Therefore $E$ is Archimedean.
\end{proof}

\begin{lemma}\label{l:nontrivialprinband}
Let $E$ be a Riesz space with $\dim(E) \geq 2$. Then $E$ contains a nontrivial principal band.
\end{lemma}

\begin{proof}
Suppose $E$ does not contain a nontrivial principal band. Then the trivial group acts principal band irreducibly on $E$,
so by Lemma~\ref{l:irrarch}, $E$ is Archimedean. Furthermore $E$ is totally ordered,
otherwise there exists an element $x$ which is neither positive nor negative and so $x^+ \notin B_{x^-} = E$.
However, by \cite[Exercise~1.14 (proven on page 272)]{aliprantesburkinshaw}, a totally ordered and Archimedean space has dimension $0$ or $1$, which is a contradiction. Hence $E$ contains a nontrivial principal band.
\end{proof}

\begin{lemma}\label{l:findimidealband}
Let $E$ be an Archimedean Riesz space and let $I \subset E$ be a finite dimensional ideal. Then $I$ is a principal projection band.
\end{lemma}

\begin{proof}
By \cite[Theorem~26.11]{riesz1} $I \cong \R^n$. Let $e_1, \ldots, e_n$ be atoms that generate $I$.
It follows that $e_1, \ldots, e_n$ are also atoms in $E$, and that $I = \sum_k I_{e_k}$, where $I_{e_k}$ denotes the ideal generated by $e_k$.
By \cite[Theorem 26.4]{riesz1} the $I_{e_k}$ are actually projection bands in $E$, and so $I$,
as a sum of principal projection bands, is a principal projection band by \cite[Chapter~4.31,~page~181]{riesz1}.
\end{proof}

\begin{lemma}\label{l:smallband}
 Let $G$ be a finite group, $E$ an Archimedean Riesz space, $B' \subset E$ a nonzero principal band and $\rho: G \to \Aut(E)$ a positive representation. Then there exists a nonzero principal band $B \subset B'$ such that for all $t \in G$, either $B \cap \rho_t B = 0$ or $B = \rho_t B$.
\end{lemma}

\begin{proof}
 The set $\cS := \{S \subset G: e \in S, \; \bigcap_{s \in S} \rho_s B' \not= 0\}$ is partially ordered by inclusion and nonempty, since $\{e\} \in \cS$. Pick a maximal element $M \in \cS$, and let $B := \bigcap_{s \in M} \rho_s B'$. Then $B$ is a principal band by \cite[Theorem~48.1]{riesz1}. Let $t \in G$ and suppose that $B \cap \rho_t B \not= 0$. Then
\[ 0 \not= B \cap \rho_t B = \bigcap_{s \in M} \rho_s B' \cap \rho_t \bigcap_{s \in M} \rho_s B' = \bigcap_{r \in M \cup tM} \rho_r B', \]
and by the maximality of $M$ we obtain $M \cup tM = M$, and so $tM \subset M$. Combined with $|tM| = |M|$ we conclude that $tM = M$, and then
\[ \rho_t B = \rho_t \bigcap_{s \in M} \rho_s B' = \bigcap_{r \in tM} \rho_r B' = \bigcap_{r \in M} \rho_r B' = B. \]
\end{proof}

Using these lemmas, we can now establish our main theorem on finite dimensionality.

\begin{theorem}\label{t:principal_band_irreducible_finite_dim}
Let $G$ be a finite group, $E$ a nonzero Riesz space and $\rho: G \to \Aut(E)$ a positive principal band irreducible representation. Then $E$ is Archimedean, finite dimensional, and the dimension of $E$ divides the order of $G$.
\end{theorem}

\begin{proof}
Lemma~\ref{l:irrarch} shows that $E$ is Archimedean. The proof is by induction on the order of $G$. If $G$ is the trivial group, then $E$ is one dimensional by Lemma~\ref{l:nontrivialprinband}, and we are done. Suppose, then, that the theorem holds for all groups of order strictly smaller than the order or $G$. If $E$ has only trivial principal bands, then by Lemma~\ref{l:nontrivialprinband} $E$ has dimension one, and we are done again. Hence we may assume that there exists a principal band $0 \neq B' \neq E$. By Lemma~\ref{l:smallband} there exists a nonzero principal band $B \subset B' \not= E$ such that $H := \{t \in G: B = \rho_t B \}$ satisfies $H^c = \{t \in G: B \cap \rho_t B = 0\}$. It is easy to see that $H$ is a subgroup of $G$, and has strictly smaller order than $G$: otherwise $B$ is a nontrivial $\rho$-invariant principal band, contradicting the principal band irreducibility of $\rho$.

We will now show that $\rho$ restricted to $H$ is principal band irreducible on the Riesz space $B$. Suppose $0 \neq A \subset B$ is an $H$-invariant principal band of $B$. By \cite[Theorem~48.1]{riesz1} $\{\sum_{s \in G} \rho_s A\}$ is a principal band, and so it is a nonzero $\rho$-invariant principal band of $E$. Hence it equals $E$, so using \cite[Theorem~20.2(ii)]{riesz1} in the second step and \cite[Exercise~7.7(iii)]{zaanen} in the third step,
\begin{align*}
B &= B \cap \left\{ \sum_{s \in G} \rho_s A \right\} \\
&= \left\{ B \cap \sum_{s \in G} \rho_s A \right\} \\
&= \left\{ \sum_{s \in G} (B \cap \rho_s A) \right\} \\
&= \left\{ \sum_{s \in H} (B \cap \rho_s A) + \sum_{s \in H^c} (B \cap \rho_s A) \right\} \\
&\subset \left\{ \sum_{s \in H} (B \cap \rho_s A) + \sum_{s \in H^c} (B \cap \rho_s B) \right\} \\
&\subset \left\{ \sum_{s \in H} (B \cap A) + \sum_{s \in H^c} 0 \right\} \\
&= \left\{ \sum_{s \in H} A \right\} \\
&= A.
\end{align*}
We conclude that $\rho|_H: H \to \Aut(B)$ is principal band irreducible, so $B$ has finite dimension by the induction hypothesis. By Lemma~\ref{l:findimidealband}, $\sum_{s \in G} \rho_s B$ is a principal band, which is nonzero and invariant, hence equal to $E$, and so $E$ has finite dimension as well.

Consider the sum $\sum_{sH \in G/H} \rho_s B$. This is well defined, since $\rho_t B = B$ for $t \in H$. Moreover, if $sH \not= rH$, then $r^{-1}s \notin H$ and so $\rho_{r^{-1}s} B \cap B = 0$, implying $\rho_s B \cap \rho_r B = 0$. Therefore $\sum_{sH \in G/H} \rho_s B$ is a sum of ideals with pairwise zero intersection, which is easily seen to be a direct sum using \cite[Theorem~17.6(ii)]{riesz1}. It follows that
\[ E = \sum_{s \in G} \rho_s B = \sum_{sH \in G/H} \rho_s B = \bigoplus_{sH \in G/H} \rho_s B. \]
Therefore
\[ \frac{|G|}{\dim(E)} = \frac{|G|}{|G : H| \dim(B)} = \frac{|H|}{\dim(B)} \in \N, \]
where the last step is by the induction hypothesis. Hence the dimension of $E$ divides the order of $G$ as well.
\end{proof}

From Theorem~\ref{t:order_dual}, where we will explicitly describe all representations as in Theorem~\ref{t:principal_band_irreducible_finite_dim}, it will also become clear that the dimension of the space divides the order of the group.

\begin{remark}\label{r:similar_theorems}
Note that Theorem~\ref{t:principal_band_irreducible_finite_dim} trivially implies a similar theorem for positive representations which are ideal irreducible, or which are band irreducible. It also answers our original question as mentioned in the Introduction: a positive projection band irreducible representation of a finite group in a Banach lattice with the projection property is finite dimensional. Indeed, an invariant principal band is then an invariant projection band.
\end{remark}

When combining Theorem~\ref{t:principal_band_irreducible_finite_dim} with Lemma~\ref{l:projband_equiv_band_irr}, we obtain the following result. Amongst others it shows that, under a mild condition on the space, various notions of irreducibility for a positive representation of a finite group are, after the fact, actually the same for finite groups. It should be compared with the equality of irreducibility and indecomposability for unitary representations of arbitrary groups, and for finite dimensional representations of finite groups whenever Maschke's theorem applies. As already mentioned earlier, if a Riesz space has sufficiently many projections, it is automatically Archimedean, cf. \cite[Theorem~30.4]{riesz1}.

\begin{theorem}\label{t:equivalencies}
 Let $G$ be a finite group, $E$ a Riesz space with sufficiently many projections and $\rho \colon G \to \Aut(E)$ a positive representation. Then the following are equivalent:
\begin{enumerate}
 \item $\rho$ is order indecomposable;
 \item $\rho$ is projection band irreducible;
 \item $\rho$ is band irreducible;
 \item $\rho$ is ideal irreducible;
 \item $\rho$ is principal band irreducible.
\end{enumerate}
If these equivalent conditions hold, then $E$ is finite dimensional, and the dimension of $E$ divides the order of $G$ if $E$ is nonzero.
\end{theorem}

\begin{proof}
 By Lemma~\ref{l:projband_equiv_band_irr} the first three conditions are equivalent. Each of the last three conditions implies that $\rho$ is principal band irreducible, so by Theorem~\ref{t:principal_band_irreducible_finite_dim} each of these three conditions implies that $E$ is finite dimensional, hence lattice isomorphic to $\R^n$ for some $n$ (\cite[Theorem~26.11]{riesz1}).  But then the collections of bands, ideals and principal bands in $E$ are all the same, and hence the last three conditions are equivalent as well. The remaining statements follow from Theorem~\ref{t:principal_band_irreducible_finite_dim}.
\end{proof}

\section{Structure of finite dimensional positive Archimedean representations}\label{s:Structure of finite dimensional representations}

Now that we know from Section~\ref{s:Irreducible and indecomposable representations} that positive representations of finite groups, irreducible as in Theorem~\ref{t:principal_band_irreducible_finite_dim} or~\ref{t:equivalencies}, are necessarily in Archimedean and finite dimensional spaces, our goal is to describe the general positive finite dimensional Archimedean representations of a finite group. In such spaces, the collections of (principal) ideals, (principal) bands and projection bands are all the same, and we will use the term ``irreducible positive representation'' throughout this section to denote the corresponding common notion of irreducibility, which is the same as order indecomposability. We will see in Theorem~\ref{t:pos_rep_sum_irr_reps} that positive finite dimensional Archimedean representations of a finite group split uniquely into irreducible positive representations. Furthermore, the order equivalence classes of finite dimensional irreducible positive representations are in natural bijection with the isomorphism classes of transitive $G$-spaces, cf.\ Theorem~\ref{t:order_dual}. The latter result can be thought of as the description of the finite dimensional Archimedean part of the order dual of a finite group. The fact that such irreducible positive representations can be realized in this way also follows from \cite[Theorem~III.10.4]{schaefer}, where it is shown that strongly continuous closed ideal irreducible positive representations of a locally compact group in a Banach lattice, with compact image in the strong operator topology, can be realized on function lattices on homogeneous spaces. This general result, however, requires considerable machinery. Therefore we prefer the method below, where all follows rather easily once an explicit description of the general, not necessarily irreducible, positive representation of a finite group in a finite dimensional Archimedean space has been obtained, a result which has some relevance of its own.

Since the decomposition result below is such a close parallel to classical semisimple representation theory of finite groups, it is natural to ask whether any other features of this purely linear context survive, such as character theory. At the end of this section we show that this is, for general groups, not the case, and in the next section we will see that this is only partly so for induction.

We now proceed towards the first main step, the description of a positive finite dimensional Archimedean representation of a finite group. Since an Archimedean finite dimensional Riesz spaces is isomorphic to $\R^n$ for some $n$ (\cite[Theorem~26.11]{riesz1}), we start by describing its group $\AutRn$ of lattice automorphisms. Naturally, the well known result \cite[Theorem~3.2.10]{meyernieberg} on lattice homomorphisms between $C_0(K)$-spaces directly implies the structure of $\AutRn$, but in this case, where $K = \{1, \ldots, n\}$, this can be seen in an elementary fashion as below. Subsequently we determine the finite subgroups of $\AutRn$. After that, we can describe the positive representations of a finite group in $\R^n$ and continue from there.

\subsection{Description of $\AutRn$}

We denote the standard basis of $\R^n$ by $\{e_1, \ldots, e_n\}$. A lattice automorphism must obviously map positive atoms to positive atoms, so each $T \in \AutRn$ maps $e_i$ to $\lambda_{ji} e_j$ for some $\lambda_{ji} > 0$. This implies that $T$ can be written uniquely as the product of a strictly positive multiplication (diagonal) operator and a permutation operator. We identify the group of permutation operators with $S_n$, so each $\sigma \in S_n$ corresponds to the operator mapping $e_i$ to $e_{\sigma(i)}$. The group of strictly positive multiplication operators is identified with $\Rn$, and so there exist unique $m \in \Rn$ and $\sigma \in S_n$ such that $T = m \sigma$.

For $\sigma \in S_n$ and $m \in \Rn$, define $\sigma(m)\in\Rn$ by $\sigma(m)_i := m_{\sigma^{-1}(i)}$. This defines a homomorphism of $S_n$ into the automorphism group of $\Rn$, hence
we can form the corresponding semidirect product $\Rn \rt S_n$, with group operation $(m_1, \sigma_1) (m_2, \sigma_2) := ( m_1 \sigma_1(m_2), \sigma_1 \sigma_2)$.
On noting that, for $i=1,\ldots,n$,
\[ \sigma m \sigma^{-1} e_i = \sigma m e_{\sigma^{-1}(i)} = \sigma m_{\sigma^{-1}(i)} e_{\sigma^{-1}(i)} = m_{\sigma^{-1}(i)} e_i = \sigma(m)_i e_i = \sigma(m) e_i, \]
it follows easily that $\chi \colon \Rn \rt S_n \to \AutRn$, defined by $\chi(m, \sigma) := m \sigma$, is a group isomorphism.
From now on we identify $\AutRn$ and $\Rn \rt S_n$, using either the operator notation or the semidirect product notation.

We let $p \colon \AutRn \to S_n$, defined by $p(m, \sigma) := \sigma$, denote the canonical homomorphism of the semidirect product onto the second factor.

\subsection{Description of the finite subgroups of $\AutRn$}

Let $G$ be a finite subgroup of $\AutRn$. Then $\ker(p|_G)$ can be identified with a finite subgroup of $\ker(p) = \Rn$. Clearly the only finite subgroup of $\Rn$ is trivial, and so $p|_G$ is an isomorphism. It follows that every finite subgroup of $\AutRn$ is isomorphic to a finite subgroup of $S_n$. The next proposition makes this correspondence explicit.

\begin{proposition}\label{p:bijection_finite_subgroups}
 Let $A$ be the set of finite subgroups $G \subset \AutRn$, and let $B$ be the set of pairs $(H, q)$, where $H \subset S_n$ is a finite subgroup and $q \colon H \to \AutRn$ is a group homomorphism such that $p \circ q = \id_H$. Define $\alpha \colon A \to B$ and $\beta \colon B \to A$ by
\[ \alpha(G) := \left( p(G), (p|_G)^{-1} \right), \quad \beta(H, q): = q(H). \]
Then $\alpha$ and $\beta$ are inverses of each other.
\end{proposition}

\begin{proof}
 Clearly $p \circ (p|_G)^{-1} = \id_{p(G)}$, so $\alpha$ is well defined. Let $G \in A$, then $\beta(\alpha(G)) = \beta(p(G), (p|_G)^{-1}) = G$. Conversely, let $(H, q) \in B$, then $\alpha(\beta(H,q)) = \alpha(q(H)) = (p(q(H)), (p|_{q(H)})^{-1} )$, and since $p \circ q = \id_H$, it follows that $p(q(H)) = H$ and that $(p|_{q(H)})^{-1} = (p|_{q(H)})^{-1} \circ p \circ q = q$.
\end{proof}

By the above proposition each finite subgroup $G$ of $\AutRn$ is determined by a subgroup $H$ of $S_n$ and a homomorphism $q \colon H \to \AutRn$ satisfying $p \circ q = \id_H$. We will now investigate such maps $q$. The condition $p \circ q = \id_H$ is equivalent with the existence of a map $f \colon H \to \Rn$, such that $q(\sigma) = (f(\sigma), \sigma)$ for $\sigma \in H$. For $\sigma, \tau \in H$, we have $q(\sigma \tau) = (f(\sigma \tau), \sigma \tau)$ and
\[ q(\sigma)q(\tau) = (f(\sigma), \sigma)(f(\tau), \tau) = (\sigma(f(\tau)) f(\sigma), \sigma \tau). \]
Hence $q$ being a group homomorphism is equivalent with $f(\sigma \tau) = \sigma(f(\tau)) f(\sigma)$ for all $\sigma, \tau \in H$, and such maps are called crossed homomorphisms.

Crossed homomorphisms of a finite group into a suitably nice abelian group $(A, +)$ (in our case $(\Rn, \cdot)$) can be characterized by the following lemma, which states, in the language of group cohomology, that $H^1(H,A)$ is trivial.

\begin{lemma}\label{l:cohomology_result}
 Let $H$ be a finite group acting on an abelian group $(A, +)$ such that, for all $a \in A$, there exists a unique element of $H$, denoted by $a/|H|$, satisfying $|H|(a/|H|) = a$. Let $f \colon H \to A$ be a map. Then $f$ is a crossed homomorphism, i.e., $f(st) = s(f(t)) + f(s)$ for all $s,t \in H$, if and only if there exists an $a \in A$ such that
\[ f(s) = a - s(a) \quad \forall s \in H. \]
\end{lemma}
\begin{proof}
Suppose $f$ is a crossed homomorphism. Let $a := \frac{1}{|H|} \sum_{r \in H} f(r)$, then, for $s \in H$,
\begin{align*}
 s(a) &= \frac{1}{|H|} \sum_{r \in H} s(f(r)) \\
&= \frac{1}{|H|} \sum_{r \in H} [f(sr) - f(s)] \\
&= \frac{1}{|H|} \sum_{r \in H} [f(r) - f(s)] \\
&= a - f(s).
\end{align*}
Hence $f(s) = a - s(a)$, as required. The converse is trivial.
\end{proof}

Combining this result with the previous discussion, we obtain the following.

\begin{corollary}
 Let $H$ be a finite subgroup of $S_n$ and let $q \colon H \to \AutRn$ be a map. Then $q$ is a homomorphism satisfying $p \circ q = \id_H$ if and only if there exists an $m \in \Rn$ such that
\[ q(\sigma) = (m \sigma(m)^{-1}, \sigma) \quad \forall \sigma \in H.\]
\end{corollary}

Rewriting this in multiplicative notation rather than semidirect product notation yields the following.

\begin{corollary}\label{c:characterization_finite_subgroups}
 Let $G$ be a finite subgroup of $\AutRn$. Then there is a unique finite subgroup $H \subset S_n$ and an $m \in \Rn$ such that
\[ G = \left\{ m \sigma(m)^{-1} \sigma: \sigma \in H \right\} = mHm^{-1}. \]
Conversely, if $H \subset S_n$ is a finite subgroup and $m \in \Rn$, then $G \subset \AutRn$ defined by the above equation is a finite subgroup of $\AutRn$.
\end{corollary}

\begin{proof}
By Proposition~\ref{p:bijection_finite_subgroups}, $G = q(p(G))$, for some $q \colon p(G) \to \Rn$ satisfying $p \circ q = \id_{p(G)}$. So $H = p(G)$ is unique, and the rest follows from the previous corollary.
\end{proof}

Note that, given a finite subgroup $G \subset \AutRn$, the subgroup $H \subset S_n$ is unique, but the multiplication operator $m$ in Corollary~\ref{c:characterization_finite_subgroups} is obviously not unique, e.g., both $m$ and $\lambda m$ for $\lambda > 0$ induce the same $G$.

\subsection{Positive finite dimensional Archimedean representations}

In this subsection we obtain our main results on finite dimensional positive representations of finite groups in Archimedean spaces: explicit description of such representations (Theorem~\ref{t:characterization_rep_into_Rn}), decomposition into irreducible positive representations (Theorem~\ref{t:pos_rep_sum_irr_reps}) and description of irreducible positive representations up to order equivalence (Theorem~\ref{t:order_dual}).

Applying the results from the previous subsection, in particular Proposition~\ref{p:bijection_finite_subgroups} and Lemma~\ref{l:cohomology_result}, we obtain the following. Recall that we view $S_n \subset \AutRn$, by identifying $\sigma \in S_n$ with a permutation matrix, and a representation $\pi \colon G \to S_n \subset \AutRn$ is called a \emph{permutation representation}. 

\begin{theorem}\label{t:characterization_rep_into_Rn}
 Let $G$ be a finite group and $\rho \colon G \to \AutRn$ a positive representation. Then there is a unique permutation representation $\pi$ and an $m \in \Rn$ such that
\[ \rho_s = m \pi_s m^{-1} \quad \forall s \in G. \]
Conversely, any permutation representation $\pi \colon G \to S_n$ and $m \in \Rn$ define a positive representation $\rho$ by the above equation.
\end{theorem}

\begin{proof}
Applying Proposition~\ref{p:bijection_finite_subgroups} to $\rho(G)$ and combining this with Lemma~\ref{l:cohomology_result}, $p \colon \rho(G) \to p \circ \rho(G)$ has an inverse of the form $q(\sigma) = m \sigma(m)^{-1} \sigma$ for some $m \in \Rn$ and all $\sigma \in p \circ \rho(G)$. We define $\pi := p \circ \rho$, then for $s \in G$,
\[ \rho_s = (q \circ p)(\rho_s) = q(\pi_s) = m \pi_s(m)^{-1} \pi_s = m \pi_s m^{-1}. \]
This shows the existence of $\pi$. The uniqueness of $\pi$ follows from the uniqueness of the factors in $\Rn$ and $S_n$ in
\[ \rho_s = m \pi_s m^{-1} = [m \pi_s(m)^{-1}] \pi_s. \]
The converse is clear.
\end{proof}

If $\rho$, $\pi$ and $m$ are related as in the above theorem, we will denote this as $\rho \sim (m, \pi)$. Note that, as in Corollary~\ref{c:characterization_finite_subgroups}, $\pi$ is unique, but $m$ is not. Given the permutation representation $\pi$, $m_1$ and $m_2$ induce the same positive representations if and only if $m_1 m_2^{-1} = \pi_s(m_1 m_2^{-1})$ for all $s \in G$.

Recall that if $\rho, \theta \colon G \to \AutRn$ are positive representations, we call them \emph{order equivalent} if there exists an intertwiner $T \in \Aut\Rn$ between $\rho$ and $\theta$. We call them \emph{permutation equivalent} if there exists an intertwiner $\sigma \in S_n$; this implies order equivalence.

\begin{proposition}\label{p:rep_isom_perm_rep}
 Let $G$ be a finite group and $\rho^1 \sim (m_1, \pi^1)$ and $\rho^2 \sim (m_2, \pi^2)$ be two positive representations of $G$ in $\R^n$. Then $\rho^1$ and $\rho^2$ are order equivalent if and only if $\pi^1$ and $\pi^2$ are permutation equivalent.
\end{proposition}

\begin{proof}
 Suppose that $\rho^1$ and $\rho^2$ are order equivalent and let $T = (m, \sigma) \in \AutRn$ be an intertwiner. Then, for all $s \in G$,
\begin{align}
\rho^1_s T &= ( m_1 \pi^1_s(m_1)^{-1}, \pi^1_s ) (m, \sigma) = ( m_1 \pi^1_s(m_1)^{-1} \pi^1_s(m), \pi^1_s \sigma ) \label{e:intertwiner1}\\
T \rho^2_s &= (m, \sigma) ( m_2 \pi^2_s(m_2)^{-1}, \pi^2_s ) = (m \sigma(m_2) \sigma \pi^2_s(m_2)^{-1} , \sigma \pi^2_s), \label{e:intertwiner2}
\end{align}
and since these are equal, $\sigma$ is an intertwiner between $\pi^1$ and $\pi^2$.

Conversely, let $\sigma$ be an intertwiner between $\pi^1$ and $\pi^2$. Then, by taking $m = m_1 \sigma(m_2)^{-1}$ and $T = (m, \sigma) \in \AutRn$, it is easily verified that, for all $s \in G$,
\[ (m_1 \pi^1_s(m_1)^{-1} \pi^1_s(m), \pi^1_s \sigma ) = (m \sigma(m_2) \sigma \pi^2_s(m_2)^{-1} , \sigma \pi^2_s), \]
 and so by \eqref{e:intertwiner1} and \eqref{e:intertwiner2}, $T$ intertwines $\rho^1$ and $\rho^2$.
\end{proof}

We immediately obtain that every positive representation is order equivalent to a permutation representation.

\begin{corollary}\label{c:rep_isom_perm_rep}
Let $G$ be a finite group and let $\rho \sim (m, \pi)$ be a positive representation of $G$ in $\R^n$. Then $\rho$ is order equivalent to $(\bf{1}, \pi)$.
\end{corollary}

\begin{remark}
The method in this subsection also yields a description of the homomorphisms from a finite group into a semidirect product $N\rt K$ with $N$ torsion-free and $H^1(H,N)$ trivial for all finite subgroups $H \subset K$, but we are not aware of a reference for this fact.
\end{remark}

Using Corollary~\ref{c:rep_isom_perm_rep}, we obtain our decomposition theorem.

\begin{theorem}\label{t:pos_rep_sum_irr_reps}
 Let $G$ be a finite group and $\rho \colon G \to \Aut(E)$ a positive representation in a nonzero finite dimensional Archimedean Riesz space $E$. Let $\{B_i\}_{i \in I}$ be the set of irreducible invariant bands in $E$. Then $E = \oplus_{i \in I} B_i$. Furthermore, any invariant band is a direct sum of $B_i$'s.
\end{theorem}

\begin{proof}
We may assume that $E=\R^n$, where bands are just linear spans of a number of standard basis vectors. By Corollary~\ref{c:rep_isom_perm_rep}, we may assume that $\rho$ is a permutation representation, which is induced by a group action on the set of basis elements. It is then clear that the irreducible invariant bands correspond to the orbits of this group action, and the invariant bands to unions of orbits. This immediately gives the decomposition of $\rho$ into irreducible positive representations, and the description of the invariant bands.
\end{proof}

We will now give a description of what could be called the finite dimensional Archimedean part of the order dual of a finite group. Note that Theorems~\ref{t:principal_band_irreducible_finite_dim} and Theorem~\ref{t:equivalencies} imply that a number of positive representations, irreducible in an appropriate way, are automatically in finite dimensional Archimedean spaces. Hence they fall within the scope of the next result, which is formulated in terms of a function lattice in order to emphasize the similarity with \cite[Theorem~III.10.4]{schaefer}.

\begin{theorem}\label{t:order_dual}
 Let $G$ be a finite group. If $H \subset G$ is a subgroup, let $(e_{tH})_{tH \in G/H}$ be the canonical basis for the finite dimensional Riesz space $C(G/H)$, defined by $e_{tH}(sH)=\delta_{tH,sH}$, for $tH,\,sH\in G/H$. Let $\pi^H \colon G \to \Aut(C(G/H))$ be the canonical positive representation corresponding to the action of $G$ on $G/H$, so that $\pi^H_s e_{tH} := e_{stH}$ for $s,t \in G$. Then, whenever $H_1$ and $H_2$ are conjugate, $\pi^{H_1}$ and $\pi^{H_2}$ are order equivalent, and the map
\[ [H] \mapsto [\pi^H] \]
is a bijection between the conjugacy classes of subgroups of $G$ and the order equivalence classes of irreducible positive representations of $G$ in nonzero finite dimensional Archimedean Riesz spaces.
\end{theorem}

\begin{proof}
It follows easily from Theorem~\ref{t:characterization_transitive_actions} that the map is well-defined.
As a consequence of Corollary~\ref{c:rep_isom_perm_rep}, every nonzero finite dimensional positive Archimedean representation is order equivalent to a permutation representation, arising from an action of $G$ on $\{1, \ldots, n\}$ for some $\geq 1$. Since the irreducibility of $\pi$ is then equivalent to the transitivity of this group action, this shows that the map is surjective. As to injectivity, suppose that $\pi^{H_1}$ and $\pi^{H_2}$ are order equivalent, for subgroups $H_1,H_2$ of $G$. Let $n=|G:H_1|=|G:H_2|$, and consider $\R^n$ with standard basis $\{e_1,\ldots,e_n\}$. Choose a bijection between the canonical basis for $C(G/H_1)$ and $\{e_1,\ldots,e_n\}$, and likewise for the canonical basis of $C(G/H_2)$. This gives a lattice isomorphism between $C(G/H_1)$ and $\R^n$, and similarly for $C(G/H_2)$. After transport of structure, $G$ has two positive representations on $\R^n$ which are order equivalent by assumption, and which originate from two permutation representations on the same set $\{1,\ldots,n\}$. As a consequence of Proposition~\ref{p:rep_isom_perm_rep}, the permutation parts of these positive representations are permutation equivalent, i.e., the two $G$-spaces, consisting of $\{1,\ldots,n\}$ and the respective $G$-actions, are isomorphic $G$-spaces. Consequently, $G/H_1$ and $G/H_2$ are isomorphic $G$-spaces, and then Lemma~\ref{t:characterization_transitive_actions} shows that $H_1$ and $H_2$ are conjugate.
\end{proof}

If $n \in \Z_{\geq 0}$ and $\rho$ is a positive representation, then $n \rho$ denotes the $n$-fold order direct sum of $\rho$. Combining the above theorem with Theorem~\ref{t:pos_rep_sum_irr_reps}, we immediately obtain the following, showing how the representations under consideration are built from canonical actions on function lattices on transitive $G$-spaces.

\begin{corollary}\label{c:decomposition_with_multiplicities}
 Let $G$ be a finite group and let $H_1, \ldots, H_k$ be representatives of the conjugacy classes of subgroups of $G$. Let $E$ be a finite dimensional Archimedean Riesz space and let $\rho \colon G \to \Aut(E)$ be a positive representation. Then, using the notation of Theorem~\ref{t:order_dual}, there exist unique $n_1, \ldots, n_k \in \Z_{\geq 0}$ such that $\rho$ is order equivalent to
\[ n_1 \pi^{H_1} \oplus \cdots \oplus n_k \pi^{H_k}. \]
\end{corollary}

\subsection{Linear equivalence and order equivalence}

If two unitary group representations are intertwined by a bounded invertible operator, they are also intertwined by a unitary operator \cite[Section~2.2.2]{dixmier}. We will now investigate the corresponding natural question in the finite dimensional ordered setting: if two positive finite dimensional Archimedean representations are intertwined by an invertible linear map, are they order equivalent? By character theory, see for example \cite[Theorem~XVIII.2.3]{lang}, two representations over the real numbers are linearly equivalent if and only if they have the same character. The following example, taken from the introduction of \cite{larsen}, therefore settles the matter.

\begin{example}\label{e:non_order_equivalence_example}
 Let $G$ be the group $\Z / 2\Z \times \Z / 2\Z$, and conser the permutation representations $\pi^1, \pi^2 \colon G \to \Aut(\R^6)$ determined by
\[ \pi^1_{(0,1)} := (12)(34) \quad \pi^1_{(1,0)} := (13)(24) \quad \pi^2_{(0,1)} := (12)(34) \quad \pi^2_{(1,0)} := (12)(56). \]
Then, as is easily verified, $\pi^1$ and $\pi^2$ have the same character, and so they are linearly equivalent. However, they are not order equivalent, since by examining the orbits of standard basis elements it follows that the first representation splits into three irreducible positive representations of dimensions 1, 1, and 4, and the second splits into three irreducible positive representations of dimension 2 each.
\end{example}

Thus, in general, linear equivalence (equivalently: equality of characters) of positive representations does not imply order equivalence. One might then try to narrow down the field: is perhaps true that two \emph{irreducible} positive representations, which are linearly equivalent, are order equivalent? In view of Theorem~\ref{t:order_dual} and Theorem~\ref{t:characterization_transitive_actions}, this is asking whether a linear equivalence of the positive representations corresponding to two transitive $G$-spaces (which is equivalent to equality, for each group element, of the number of fixed points in the two spaces) implies that these $G$-spaces are isomorphic. The answer, again, is negative, but counterexamples are now more intricate to construct than above, and the reader is referred to \cite[Theorem~1]{desmitlenstra}, providing an abundance of such counterexamples.

On the positive side, in some cases linear equivalence of irreducible positive representations does imply order equivalence, as shown by the next result.

\begin{lemma}
 Let $G$ be a finite group, let $N$ be a normal subgroup, and let $\pi^N \colon G \to \Aut(C(G/N))$ be the irreducible positive representations as in Theorem~\ref{t:order_dual}. Then an irreducible positive representation which is linearly equivalent with $\pi^N$ is in fact order equivalent with $\pi^N$.
\end{lemma}

\begin{proof}
Passing to an order equivalent model we may, in view of Theorem~\ref{t:order_dual}, assume that the other irreducible positive representation is $\pi^H$, for a subgroup $H$ of $G$. The fact that $G/N$ is a group implies easily that the character of $\pi^N$ equals $|G:N|\mathbf{1}_N$. Since the character of $\pi^H$, which is equal to that of $\pi^N$ by their linear equivalence, is certainly nonzero on $H$, we see that $H\subset N$. On the other hand, equality of dimensions yields $|G:N|=|G:H|$, hence $|N|=|H|$. We conclude that $H = N$.
\end{proof}

Combining Theorem~\ref{t:order_dual} with the previous lemma yields the following.

\begin{corollary}\label{c:dedekindgroups}
 Let $G$ be a finite group with only normal subgroups. If two finite dimensional irreducible positive representations of $G$ are linearly equivalent (equivalently: have the same character), they are order equivalent.
\end{corollary}

Thus, for such groups (so-called Dedekind groups), and in particular for abelian groups, the classical correspondence between characters and irreducible representations survives in an ordered context---where, naturally, ``irreducible" has a different meaning. However, as Example~\ref{e:non_order_equivalence_example} shows, already for abelian groups this correspondence breaks down for reducible positive representations.

\section{Induction}\label{s:Induction}

In this section we will examine the theory of induction in the ordered setting from a categorical point of view. It turns out that the results are to a large extent analogous to the linear case as covered in many sources, e.g., \cite[Section~XVIII.7]{lang}. Still, it seems worthwhile to go through the motions, as a preparation for future more analytical constructions, and in doing so we then also obtain a slight bonus (the original ordered module is embedded in the induced one as a sublattice, even though this was not required), keep track of several notions of irreducibility, and also observe that Frobenius reciprocity holds only partially. Since we do not consider topological issues at the moment, we are mostly interested in the case where the group is finite, but the theory is developed at little extra cost in general for arbitrary groups and subgroups. Our approach is thus slightly more general than, e.g., the approach in \cite{lang}, as we do not require our groups to be finite or our subgroups to be of finite index.

For the rest of the section, $G$ is a not necessarily finite group, $H$ is a subgroup of $G$, not necessarily finite or of finite index, $R$ is a system of representatives of $G / H$, and the Riesz spaces are not assumed to be finite dimensional. The only finiteness condition is in Corollary \ref{c:induced_by_primitive}, where $G$ is assumed to be finite.

\subsection{Definitions and basic properties}

A pair $(E, \rho)$, where $E$ is a Riesz space and $\rho \colon G \to \Aut(E)$ is a positive representation, is called an \emph{ordered $G$-module}. In this notation, we will often omit the representation $\rho$. If $E$ is an ordered $G$-module, it is also an ordered $H$-module by restricting the representation to $H$. If $(E, \rho)$ and $(F, \theta)$ are ordered $G$-modules, then the positive cone of positive intertwiners between $\rho$ and $\theta$ will be denoted by $\Hom_G(E, F)$. Two ordered $G$-modules are \emph{isomorphic ordered $G$-modules} if there exists an intertwining lattice isomorphism.

\begin{defn}\label{d:def_induced_rep}
 Let $F$ be an ordered $H$-module. A pair $(\Ind^G_H(F), j)$, where $\Ind^G_H(F)$ is an ordered $G$-module and $j \in \Hom_H(F, \Ind^G_H(F))$ is called an \emph{induced ordered module} of $F$ from $H$ to $G$ if it satisfies the following universal property:

For any ordered $G$-module $E$ and $T \in \Hom_H(F, E)$, there is a unique $\overline{T} \in \Hom_G(\Ind^G_H(F), E)$ such that $T = \overline{T} \circ j$, i.e., such that the following diagram is commutative:
\[ \xymatrix
{
F \ar[rd]_j \ar[rr]^T &  & E \\
 & \Ind^G_H(F) \ar[ru]_{\overline{T}} & } \]
If $\theta$ is the positive representation of $H$ turning $F$ into an ordered $H$-module, then the positive representation of $G$ turning $\Ind^G_H(F)$ into an ordered $G$-module will be denoted by $\Ind_H^G(\theta)$ and will be called the \emph{induced positive representation} of $\theta$ from $H$ to $G$.
\end{defn}

First we will show, by the usual argument, that the induced ordered module is unique up to isomorphism of ordered $G$-modules.

\begin{lemma}\label{l:induced_unique}
 Let $F$ be an ordered $H$-module and let $(E_1, j_1)$ and $(E_2, j_2)$ be induced ordered modules of $F$ from $H$ to $G$. Then $E_1$ and $E_2$ are isomorphic as ordered $G$-modules.
\end{lemma}

\begin{proof}
 Using the universal property, we obtain the unique maps $\overline{j_1} \in \Hom_G(E_2, E_1)$ satisfying $j_1 = \overline{j_1} \circ j_2$ and $\overline{j_2} \in \Hom_G(E_1, E_2)$ satisfying $j_2 = \overline{j_2} \circ j_1$. It follows that
\[ j_1 = \overline{j_1} \circ j_2 = \overline{j_1} \circ \overline{j_2} \circ j_1. \]
 Now consider the ordered $G$-module $E_1$, and apply its universal property to itself - we obtain the \emph{unique} map $\id_{E_1} \in \Hom_G(E_1, E_1)$ such that $j_1 = \id_{E_1} \circ j_1$. Together with the above equation, this shows that $\overline{j_1} \circ \overline{j_2} = \id_{E_1}$. Similarly we obtain $\overline{j_2} \circ \overline{j_1} = \id_{E_2}$, and so $\overline{j_2}$ is an isomorphism of ordered $G$-modules.
\end{proof}

We will now construct the induced ordered module, which is the usual induced module, but now with an additional lattice structure. Let $(F, \theta)$ be an ordered $H$-module. We define the ordered vector space
\begin{equation}\label{e:definition_induced_space}
 \tilde{E} := \{ f: G \to F| \; f(st) = \theta_{t^{-1}} f(s) \; \forall s \in G, \; \forall t \in H \},
\end{equation}
with pointwise ordering. Using that $\theta_{t^{-1}}$ is a lattice isomorphism for $t \in H$, one easily verifies that $\tilde{E}$ is a Riesz space, with pointwise lattice operations. Furthermore, if $S \subset G$ is a subset, then we define the subset
\begin{equation}\label{e:definition_induced_space_ideal}
 E_S := \{ f \in \tilde{E}\; | \; \supp(f) \subset S \}.
\end{equation}
Let $\rho: G \to \Aut(\tilde{E})$ be defined by $(\rho_s f)(u) := f(s^{-1}u)$, for $s,u \in G$. Then $\rho$ is a positive representation turning $\tilde{E}$ into an ordered $G$-module. Moreover, for $s \in G$, $\supp(\rho_s f) = s \cdot \supp(f)$, so $\rho_s E_H = E_{sH}$. Now we define the $\rho$-invariant Riesz subspace
\begin{equation}\label{e:induced_direct_sum}
E := \bigoplus_{r \in R} E_{rH} = \bigoplus_{r \in R} \rho_r E_H \subset \tilde{E}.
\end{equation}
For $x \in F$, define a function $j(x) \colon G \to F$ by $j(x)(t) := \theta_{t^{-1}} x$ for $t \in H$, and $j(x)(t) := 0$ for $t \notin H$. It is routine to verify that $j(x) \in E_H$ for all $x \in F$, and that $j \colon F \to E_H$ is an isomorphism of ordered $H$-modules between $(F, \theta)$ and $(E_H, \rho)$. This last fact and the fact that $E = \oplus_{r \in R} \rho_r F$ as an order direct sum will be used to prove that $(E,j)$ actually satisfies the desired universal property.


\begin{lemma}\label{l:induced_module}
Let $(E, \rho)$ be an ordered $G$-module and let $F$ be a Riesz subspace of $E$ which is invariant under the restricted representation $\theta = \rho|_H$ of $H$ in $E$, and such that $E = \oplus_{r \in R} \rho_r F$ as an order direct sum. Let $j \colon F \to E$ be the embedding. Then $(E, j)$ is an induced ordered module of $F$ from $H$ to $G$.
\end{lemma}

\begin{proof}
Let $(E', \rho')$ be another ordered $G$-module, and let $T: F \to E'$ be a positive linear map such that $T (\theta_t x) = \rho'_t T(x)$ for all $t \in H$ and $x \in F$. We have to show that there exists a unique positive linear map $\overline{T}: E \to E'$ extending $T$ and satisfying $\overline{T} \circ \rho_s = \rho_s' \circ \overline{T}$ for all $s \in G$.

We follow the proof of \cite[Lemma~3.3.1]{serre}. If $\overline{T}$ satisfies these conditions, and if $x \in \rho_r F$ for $r \in R$, then $\rho_r^{-1} x \in F$, and so
\[ \overline{T}(x) = \overline{T}(\rho_r \rho_r^{-1} x) = \rho_r' \overline{T}(\rho_r^{-1} x) = \rho_r' T(\rho_r^{-1} x). \]
This formula determines $\overline{T}$ on $\rho_r F$, hence on $E$ since it is the direct sum of the $\rho_r F$. This proves the uniqueness of $\overline{T}$.

For the existence, let $x \in \rho_r F$, then we define $\overline{T}(x) := \rho_r' T(\rho_r^{-1} x)$ as above. This does not depend on the choice of representative $r$, since if we replace $r$ by $rt$ with $t \in H$, we have
\[ \rho'_{rt} T (\rho^{-1}_{rt} x) = \rho_r' \rho_t' T(\theta_t^{-1} \rho_r^{-1} x) = \rho_r' T(\theta_t \theta_t^{-1} \rho_r^{-1} x) = \rho_r' T(\rho_r^{-1} x). \]
Since $E$ is the direct sum of the $\rho_r F$, there exists a unique linear map $\overline{T} \colon E \to E'$ which extends the partial mappings thus defined on the $\rho_r F$. It is easily verified that $\overline{T} \rho_s = \rho_s' T$ for all $s \in G$. Since all mappings involved are positive, $\overline{T}$ is positive as well.
\end{proof}

\begin{corollary}
 Let $F$ be an ordered $H$-module. Then the induced ordered module $(\Ind_H^G(F),j)$ of $F$ from $H$ to $G$ exists, and $\Ind_H^G(F)$ is unique up to isomorphism of ordered $G$-modules. The positive map $j$ is actually an injective lattice homomorphism of the ordered $H$-module $F$ into the ordered $H$-module $\Ind_H^G(F)$. Moreover, if $E$ is finite dimensional and $H$ has finite index $|G : H|$, then $\textup{dim}(\Ind_H^G(F)) = |G : H|\, \textup{dim}(E)$.
\end{corollary}

\begin{proof}
 The existence follows from Lemma~\ref{l:induced_module} and the construction preceding it, which also shows that $j$ has the property as described. The uniqueness follows from Lemma~\ref{l:induced_unique}. The last statement follows from \eqref{e:induced_direct_sum}.
\end{proof}

We continue with some properties of the induced positive representation.

\begin{proposition}\label{p:induced_irreducible_implies_irreducible}
 Let $\theta$ be a positive representation of a subgroup $H \subset G$. If $\Ind^G_H(\theta)$ is either band irreducible, or ideal irreducible, or projection band irreducible, then so is $\theta$.
\end{proposition}

\begin{proof}
We will prove this for ideals, the other cases are identical. Let $E$ be as in \eqref{e:definition_induced_space}, \eqref{e:definition_induced_space_ideal} and \eqref{e:induced_direct_sum}. Suppose $\theta$ is not ideal irreducible, so there exists a proper nontrivial $\theta$-invariant ideal $B \subset E_\theta$. Then $\{ f \in E: f(G) \subset B\} \subset E$ is a proper nontrivial $\Ind^G_H(\theta)$-invariant ideal, so $\Ind^G_H(\theta)$ is not ideal irreducible.
\end{proof}

\begin{proposition}[Induction in stages]\label{p:induction_in_stages}
Let $H \subset K \subset G$ be a chain of subgroups of $G$, and let $F$ be an ordered $H$-module. Then
\[ (\Ind_K^G(\Ind_H^K(F)), j_K^G \circ j_H^K) \]
is an induced ordered module of $F$ from $H$ to $G$.
\end{proposition}

\begin{proof}
 Let $E$ be an ordered $G$-module. Consider the following diagram:
\[ \xymatrix{
 F \ar[r]^{j_H^K} \ar[rd]^{T} & \Ind_H^K(F) \ar[r]^{j_K^G} \ar[d]^{\overline{T}} & \Ind_K^G(\Ind_H^K(F)) \ar[ld]^{\overline{\overline{T}}} \\
 & E & }
\]
Here $\overline{T}$ is the unique positive map generated by $T$, and $\overline{\overline{T}}$ is the unique positive map generated by $\overline{T}$. Since the diagram is commutative, $T = \overline{\overline{T}} \circ (j_K^G \circ j_H^K)$. If $S$ is another positive map satisfying $T = S \circ (j_K^G \circ j_H^K)$, then $(S \circ j_K^G) \circ j_H^K = T = \overline{T} \circ j_H^K$, and so $S \circ j_K^G = \overline{T}$ by the uniqueness of $\overline{T}$. This in turn implies that $S = \overline{\overline{T}}$ by the uniqueness of $\overline{\overline{T}}$, and so $(\Ind_K^G(\Ind_H^K(F)), j_K^G \circ j_H^K)$ satisfies the universal property, as desired.
\end{proof}

\subsection{Frobenius reciprocity}

This subsection is concerned with the implications, or rather their absence, of the functorial formulation of Frobenius reciprocity for multiplicities of irreducible positive representations in induced ordered modules.

We start with the usual consequence of the categorical definition of the induced ordered module: induction from $H$ to $G$ is the left adjoint functor of restriction from $G$ to $H$, for an arbitrary group $G$ and subgroup $H$.

\begin{proposition}[Frobenius Reciprocity]\label{p:frobenius}
Let $F$ be an ordered $H$-module and let $E$ be an ordered $G$-module. Then there is a natural isomorphism of positive cones
\[ \Hom_H(F, E) \cong \Hom_G(\Ind_H^G(F), E). \]
\end{proposition}

\begin{proof}
The existence of the natural bijection is an immediate consequence of the very definition of the induced module in Definition~\ref{d:def_induced_rep}. For $\lambda, \mu \geq 0$ and $T,S\in \Hom_H(F, E)$ we have $\overline{\lambda T + \mu S} = \lambda \overline{T} + \mu \overline{S}$ as a consequence of the uniqueness part of Definition~\ref{d:def_induced_rep} and the positivity of $\lambda \overline{T} + \mu \overline{S}$. Hence the two positive cones are isomorphic.
\end{proof}

Now suppose $G$ is a finite group. For finite dimensional positive Archimedean representations of $G$ we have a unique decomposition into irreducible positive representations, according to Corollary~\ref{c:decomposition_with_multiplicities}. Hence the notion of multiplicity is available, and if $\rho_1$ is a finite dimensional positive representation and $\rho_2$ a finite dimensional irreducible positive representation of $G$, we let $m(\rho_1, \rho_2)$ denote the number of times that a lattice isomorphic copy of $\rho_1$ occurs in the decomposition of $\rho_2$ of Corollary~\ref{c:decomposition_with_multiplicities}. Now Let $\theta$ be a finite dimensional irreducible positive representation of a subgroup $H \subset G$ and let $\rho$ be a finite dimensional irreducible positive representation of $G$. In view of the purely linear theory, cf. \cite[Proposition~21]{serre}, and its generalization to unitary representations of compact groups, cf. \cite[Theorem~5.9.2]{wolf}, it is natural to ask whether
\[ m(\rho, \Ind_H^G(\theta)) = m(\theta, \rho|_H) \]
still holds in our ordered context. In the linear theory, and also for compact groups, this follows from the fact that the dimensions of spaces of intertwining operators in the analogues of Proposition~\ref{p:frobenius} can be interpreted as a multiplicities. Since the sets in Proposition~\ref{p:frobenius} are cones and not vector spaces, and their elements are not even necessarily lattice homomorphisms, there seems little chance of success in our case. Indeed, Frobenius reciprocity  in terms of multiplicities does not hold for ordered modules, as is shown by the following counterexample.
We let $\theta \colon H=\{e\} \to \Aut(\R)$ be the trivial representation; then $\Ind_H^G(\theta)$ is the left regular representation of $G$ on the Riesz space with atomic basis $\{e_s\}_{s \in G}$. This set of basis elements has only one $G$-orbit, hence $\Ind_H^G(\theta)$ is band irreducible. Therefore, if  $\rho$ is an arbitrary irreducible positive representation of $G$, $m(\rho, \Ind_H^G(\theta))$ is at most one. On the other hand, $\rho|_H$ decomposes as $\dim \rho$ copies of the trivial representation $\theta$, so $m(\theta, \rho|_H) = \dim \rho$.

Another counterexample, where $H$ is nontrivial, can be obtained by taking $G = \Z / 4\Z$ and $H = \{0, 2\}$, and taking $\theta$ and $\rho$ to be the left regular representation of $H$ and $G$, respectively. Then these are irreducible positive representations of the respective groups, and it can be verified that $\Ind_H^G(\theta) \cong \rho$, so $m(\rho, \Ind_H^G(\theta)) = 1$, but $\rho|_H \cong \theta \oplus \theta$, so $m(\theta, \rho|_H) = 2$.

\subsection{Systems of imprimitivity}

In this final subsection, we consider systems of imprimitivity in the ordered setting. As before, $G$ is an arbitrary group and $H \subset G$ an arbitrary subgroup. We start with an elementary lemma, which is easily verified.

\begin{lemma}
 Let $E$ and $F$ be Riesz spaces and let $T \colon E \to F$ be a lattice isomorphism. If $B \subset E$ is a projection band in $E$, then $TB$ is a projection band in $F$, and the corresponding band projections are related by $P_{TB} = T P_B T^{-1}$.
\end{lemma}
%

Let $\rho \colon G \to \Aut(E)$ be a positive representation. Suppose there exists a $G$-space $\Gamma$, and a family of Riesz subspaces $\{E_\gamma\}_{\gamma \in \Gamma}$, such that $E = \oplus_{\gamma \in \Gamma} E_\gamma$ as an order direct sum and $\rho_s E_\gamma = E_{s \gamma}$ for all $s \in G$. Then we call the family $\{E_\gamma\}_{\gamma \in \Gamma}$ an \emph{ordered system of imprimitivity} for $\rho$. If $A \subset \Gamma$, then the order decomposition
\[ E = \left( \bigoplus_{\gamma \in A} E_\gamma \right) \oplus \left(\bigoplus_{\gamma \in A^c} E_\gamma \right) \]
implies that $\oplus_{\gamma \in A} E_\gamma$ is a projection band. In particular, $E_\gamma$ is a projection band for all $\gamma \in \Gamma$. For a subset $A \subset \Gamma$, let $P(A)$ denote the band projection $P_{\oplus_{\gamma \in A} E_\gamma}$ onto $\oplus_{\gamma \in A} E_\gamma$. Then the assignment $A \mapsto P(A)$ is a band projection valued map satisfying
\begin{equation}\label{e:projection_valued_map}
 P \left( \bigcup_{i \in I} A_i \right)x = \sup_{i \in I} P(A_i)x 
\end{equation}
for all $x \in E^+$ and all collections of subsets $\{A_i\}_{i \in I} \subset \Gamma$. An equivalent formulation of \eqref{e:projection_valued_map} is $P(\sup_i A_i) = \sup_i P(A_i)$, where the first supremum is taken in the partially ordered set of subsets of $\Gamma$, and the second supremum is taken in the partially ordered space of regular operators on $E$. Either formulation is the ordered analogue of a strongly $\sigma$-additive spectral measure. Furthermore, the map $P$ is covariant in the sense that, for $s \in G$,
\[ P( s A) = P_{\oplus_{\gamma \in A} E_{s \gamma}}= P_{\rho_s \oplus_{\gamma \in A} E_\gamma} = \rho_s \left( P_{\oplus_{\gamma \in A} E_\gamma} \right) \rho_s^{-1} = \rho_s P(A) \rho_s^{-1}, \]
where the above lemma is used in the penultimate step. Every positive representation admits a system of imprimitivity where $\Gamma$ has exactly one element, and such a system of imprimitivity will be called trivial. A system of imprimitivity is called transitive if the action of $G$ on $\Gamma$ is transitive.

\begin{defn}
 A positive representation $\rho$ is called \emph{primitive} if it admits only the trivial ordered system of imprimitivity.
\end{defn}

Every decomposition of $E$ into $\rho$-invariant projection bands corresponds to an ordered system of imprimitivity where the action of $G$ on $\Gamma$ is trivial, so $\rho$ is projection band irreducible if and only if $\rho$ admits no nontrivial ordered system of imprimitivity with a trivial action. Hence a primitive positive representation is projection band irreducible, i.e., order indecomposable.

\begin{theorem}[Imprimitivity Theorem]\label{t:imprimitivity_theorem}
 Let $\rho$ be a positive representation of $G$. The following are equivalent:
\begin{enumerate}
 \item $\rho$ admits a nontrivial ordered transitive system of imprimitivity;
 \item There exists a proper subgroup $H \subset G$ and a positive representation $\theta$ of $H$ such that $\rho$ is order equivalent to $\Ind_H^G(\theta)$.
\end{enumerate}
\end{theorem}

\begin{proof}
 $(ii) \Rightarrow (i)$: Suppose that $\rho$ is order equivalent to $\Ind_H^G(\theta)$. Let $\Gamma$ be the transitive $G$-space $G / H$, which is nontrivial because $H$ is proper, and consider the spaces $E$ and $\{E_{sH}\}_{sH \in \Gamma}$ defined in \eqref{e:definition_induced_space}, \eqref{e:definition_induced_space_ideal} and \eqref{e:induced_direct_sum}. By the discussion following these definitions, $E = \oplus_{sH \in \Gamma} E_{sH}$, and $\rho_u E_{sH} = E_{usH}$, so this defines a nontrivial transitive system of imprimitivity.

$(i) \Rightarrow (ii)$: Suppose $\rho$ admits a nontrivial transitive ordered system of imprimitivity. Then by Theorem~\ref{t:characterization_transitive_actions} we may assume $\Gamma = G / H$ for some subgroup $H$, which must be proper since $\Gamma$ is nontrivial. Choose a system of representatives $R$ of $G / H$, then we may assume $\Gamma = R$. Assume that the representative of $H$ in $R$ is $e$. We have that $E = \oplus_{r \in R} E_r$, and if $t \in H$, then $t$ acts trivially on $e \in R$, so $\rho_t E_e = E_e$. Therefore we can define $\theta \colon H \to \Aut(E_e)$ by restricting $\rho$ to $H$ and letting it act on $E_e$. Then by the definition of the system of imprimitivity $\{E_r\}_{r \in R}$ we have $\rho_r E_e = E_r$, and so
\[ \bigoplus_{r \in R} \rho_r E_e = \bigoplus_{r \in R} E_r = E, \]
which implies by Lemma~\ref{l:induced_module} that $\rho$ is induced by $\theta$.
\end{proof}

\begin{corollary}\label{c:induced_by_primitive}
All projection band irreducible positive representations of a finite group $G$ are induced by primitive positive representations.
\end{corollary}

\begin{proof}
 Let $\rho$ be a projection band irreducible representation. If $\rho$ is primitive we are done, so assume it is not primitive. Then there exists a nontrivial ordered system of imprimitivity $\{E_\gamma\}_{\gamma \in \Gamma}$. Then for each orbit in $\Gamma$, the direct sum of $E_\gamma$, where $\gamma$ runs through the orbit, is a $\rho$-invariant projection band. Since $\rho$ is projection band irreducible, this implies that $\Gamma$ is transitive. Therefore, by the Imprimitivity Theorem~\ref{t:imprimitivity_theorem}, $\rho$ is induced by a positive representation of a proper subgroup of $G$, which is projection band irreducible by Proposition~\ref{p:induced_irreducible_implies_irreducible}. If this representation is primitive we are done, and if not we keep repeating the process until a representation is induced by a primitive positive representation. Then by Proposition~\ref{p:induction_in_stages} the representation $\rho$ is induced by this primitive positive representation as well.
\end{proof}

We note that, in the above corollary, the representations need not be finite dimensional, and that it trivially implies a similar statement for band irreducible and ideal irreducible positive representations of a finite group.

\subsection*{Acknowledgements}
The authors thank Lenny Taelman for helpful comments and suggestions.

\bibliographystyle{amsplain}
\bibliography{de_jeu_wortel_finite_groupsbib}

\providecommand{\bysame}{\leavevmode\hbox to3em{\hrulefill}\thinspace}
\providecommand{\MR}{\relax\ifhmode\unskip\space\fi MR }
\providecommand{\MRhref}[2]{%
  \href{http://www.ams.org/mathscinet-getitem?mr=#1}{#2}
}
\providecommand{\href}[2]{#2}
\begin{thebibliography}{10}

\bibitem{aliprantesburkinshaw}
C.D. Aliprantis and O.~Burkinshaw, \emph{Locally solid {R}iesz spaces with
  applications to economics}, second ed., vol. 105, American Mathematical
  Society, Providence, RI, 2003.

\bibitem{dixmier}
J.~Dixmier, \emph{{$C\sp*$}-algebras}, North-Holland Publishing Co., Amsterdam,
  1977.

\bibitem{dejeumesserschmidt}
M.F.E.~de Jeu and H.J.M. Messerschmidt, \emph{A {J}ordan-{H}\"older theorem for
  {R}iesz spaces}, to appear.

\bibitem{dejeurozendaal}
M.F.E.~de Jeu and J.~Rozendaal, \emph{Decomposing positive representations in
  ${L}^p$-spaces for {P}olish transformation groups}, to appear, cf.\
  http://www.math.leidenuniv.nl/en/theses/240/.

\bibitem{lang}
S.~Lang, \emph{Algebra}, third ed., Graduate Texts in Mathematics, vol. 211,
  Springer-Verlag, New York, 2002.

\bibitem{larsen}
M.~Larsen, \emph{On the conjugacy of element-conjugate homomorphisms. {II}},
  Quart. J. Math. Oxford Ser. (2) \textbf{47} (1996), no.~185, 73--85.

\bibitem{riesz1}
W.A.J. Luxemburg and A.C. Zaanen, \emph{Riesz spaces. {V}ol. {I}},
  North-Holland Publishing Co., Amsterdam, 1971.

\bibitem{meyernieberg}
P.~Meyer-Nieberg, \emph{Banach lattices}, Universitext, Springer-Verlag,
  Berlin, 1991.

\bibitem{schaefer}
H.H. Schaefer, \emph{Banach lattices and positive operators}, Springer-Verlag,
  New York, 1974.

\bibitem{serre}
J-P. Serre, \emph{Linear representations of finite groups}, Springer-Verlag,
  New York, 1977.

\bibitem{desmitlenstra}
B.~de Smit and H.W. Lenstra, Jr., \emph{Linearly equivalent actions of solvable
  groups}, J. Algebra \textbf{228} (2000), no.~1, 270--285.

\bibitem{counterexamples}
L.A. Steen and J.A. Seebach, Jr., \emph{Counterexamples in topology}, Dover
  Publications Inc., Mineola, NY, 1995.

\bibitem{wolf}
J.A. Wolf, \emph{Harmonic analysis on commutative spaces}, vol. 142, American
  Mathematical Society, Providence, RI, 2007.

\bibitem{zaanen}
A.C. Zaanen, \emph{Introduction to operator theory in {R}iesz spaces},
  Springer-Verlag, Berlin, 1997.

\end{thebibliography}

\end{document}